\documentclass[12pt]{tac}

\usepackage{amsmath, amssymb}
\usepackage{enumerate}
\usepackage[all]{xy}

\input diagxy

\renewcommand{\S}{{\mathsf{(S)}}}
\newcommand{\T}{{\mathsf{(T)}}}
\newcommand{\Ac}{{\mathcal{A}}}
\newcommand{\Bc}{{\mathcal{B}}}
\newcommand{\Pt}{{\mathsf{Pt}}}

\newcommand{\Com}{{\mathsf{Com}}}

\newcommand{\Set}{{\mathsf{Set}}}

\newcommand{\CMon}{{\mathsf{CMon}}}

\newcommand{\Grp}{{\mathsf{Grp}}}
\newcommand{\Ab}{{\mathsf{Ab}}}
\newcommand{\Z}{{\mathsf{Z}}}
\newcommand{\ZZ}{{\mathsf{\Sigma}}}
\newcommand{\C}{\ensuremath{\mathbb{C}}}
\newcommand{\V}{\ensuremath{\mathbb{V}}}
\newcommand{\Eq}{\ensuremath{\mathsf{Eq}}}
\renewcommand{\lim}{\mathrm{lim}}
\newcommand{\colim}{\mathrm{colim}}
\renewcommand{\:}{\colon}

\title{Centrality and the commutativity of finite products with coequalisers}

\author{Michael Hoefnagel}

\eaddress{mhoefnagel@sun.ac.za}

\address{Department of Mathematical Sciences \\ Stellenbosch University\\
	Private Bag X1 \\
	7602, Matieland\\
	South Africa  \\ \\ 
 National Institute for Theoretical and Computational Sciences (NITheCS) \\ South Africa
}



\copyrightyear{2023}

\begin{document}

\maketitle

\begin{abstract}
We study centrality of morphisms in a setting derived from that of a pointed category in which finite products commute with coequalisers. The main results of this paper show that much of the behaviour of central morphisms for unital categories \cite{Bou96} is retained in our setting, including categories which are (weakly) unital, but also categories outside of the unital setting.
\end{abstract}



\section{Introduction} \label{section: introduction}

Recall that given small categories $\Ac,\Bc$ and a category $\C$ which has limits of shape $\Ac$ as well as colimits of shape $\Bc$, every bifunctor $F\: \Ac \times \Bc \rightarrow \C$ determines two functors $\colim_{\Bc}F \: \Ac \rightarrow \C$ and $\lim_{\Ac}F\:\Bc \rightarrow \C$ where for each object $A$ of $\Ac$ and $B$ of $\Bc$ we have $\colim_{\Bc}F(A) = \colim F(A,-)$ and $\lim_{\Ac}F(B) = \lim F(-,B)$. For morphisms $f\:A \rightarrow A'$ in $\Ac$ and $g\: B \rightarrow B'$ in $\Bc$ the corresponding natural transformations $F(f,-)\colon F(A,-) \rightarrow F(A',-)$ and $F(-,g)\colon F(-,B) \rightarrow F(-,B')$ determine the morphisms $\colim_{\Bc}F(f)$ and $\lim_{\Ac}F (g)$ respectively. These functors, in turn, determine a limit $\lim (\colim_{\Bc} F)$ and a colimit $\colim (\lim_{\Ac} F)$, and between them is a canonically induced morphism $\omega_F\: \colim (\lim_{\Ac} F) \rightarrow \lim (\colim_{\Bc} F)$. Limits of shape $\Ac$ are said to \emph{commute} with colimits of shape $\Bc$ in $\C$ if the morphism $\omega_F$ is an isomorphism for any bifunctor $F\: \Ac \times \Bc \rightarrow \C$. It is well-known, for instance, that filtered colimits commute with all finite limits in the category $\Set$ (see \cite{Mac98}). That is limits of shape $\Ac$ commute with colimits of shape $\Bc$ when $\Ac$ is finite and $\Bc$ is filtered. The larger class of sifted colimits \cite{AdamekRosicky2000, AdamekRosickyVitale2010} are such that they commute with all finite products in $\Set$, so that reflexive coequalisers, for instance, commute with finite products in $\Set$, and more generally in any variety of algebras. Arbitrary coequalisers, however, need not commute with finite products for general varieties of algebras. However, in many special classes of varieties (such as the variety of groups or monoids), we do have commutativity of finite products and arbitrary coequalisers (see \cite{Hoefnagel2019b} for examples). 

Considered as a property of a category, the commutativity of a specified limit with a specified colimit can have significant consequences for the underlying category. We give two illustrations of this: suppose that $\C$ is a category with finite products and finite coproducts and suppose that binary products commute with binary coproducts. That is, if $X_1 \xrightarrow{x_1} X \xleftarrow{x_2} X_2$ and $Y_1 \xrightarrow{y_1} Y \xleftarrow{y_2} Y_2$ are coproduct diagrams in $\C$ then the diagram
\[
X_1 \times Y_1 \xrightarrow{x_1 \times y_1} X \times Y \xleftarrow{x_2 \times y_2} X_2 \times  Y_2
\]
is a coproduct diagram in $\C$. The product of the trivial coproducts $1 \rightarrow 1 \leftarrow 0$ and $0 \rightarrow 1 \leftarrow 1$ being a coproduct gives us that $\C$ is pointed. Then, the product of the trivial coproducts $X \rightarrow X \leftarrow 0$ and $0 \rightarrow Y \leftarrow Y$ yield a biproduct of $X$ and $Y$ in $\C$, so that $\C$ is \emph{linear} (or ``half-additive'' \cite{FreydScendrov1990}). Conversely, in every linear category finite products commute with finite coproducts in $\C$.  Now consider a category $\C$ with equalisers and coequalisers, and suppose that equalisers commute with coequalisers in $\C$, i.e., that in any diagram
\[
\xymatrix{
E_1 \ar@<-.5ex>[d]_{e_1}  \ar@<.5ex>[d]^{e_2} \ar[r] & X_1 \ar@<-.5ex>[d]_{x_1} \ar@<.5ex>[d]^{x_2} \ar@<-.5ex>[r] \ar@<.5ex>[r]& Y_1 \ar@<-.5ex>[d]_{y_1} \ar@<.5ex>[d]^{y_2}\\
E_2 \ar[d] \ar[r] & X_2  \ar[d] \ar@<-.5ex>[r] \ar@<.5ex>[r]& Y_2 \ar[d] \\
U \ar[r]& V \ar@<-.5ex>[r] \ar@<.5ex>[r] & W
}
\]
where the top and middle rows are equaliser diagrams and $(e_1,x_1,y_1)$ and $(e_2,x_2,y_2)$ are natural transformations between the equaliser diagrams, if all of the columns are coequaliser diagrams then the bottom row is an equaliser diagram (where the morphisms are all the canonically induced morphisms, making bottom right square reasonably commutative). Given two parallel morphisms $f,g \colon X \rightarrow Y$ in $\C$, we form
\[
\xymatrix{
E \ar@<-.5ex>[d]_{fe} \ar@<.5ex>[d]^{ge} \ar[r]^e & X \ar@<-.5ex>[d]_{f} \ar@<.5ex>[d]^{g} \ar@<-.5ex>[r]_g \ar@<.5ex>[r]^f & Y \ar@<-.5ex>[d]_q\ar@<.5ex>[d]^q\\
Y \ar[d]_{1} \ar[r]_1 & Y  \ar[d]^q \ar@<-.5ex>[r]_q \ar@<.5ex>[r]^q & Q \ar[d]^1 \\
Y \ar[r]_q & Q \ar@<-.5ex>[r]_1 \ar@<.5ex>[r]^1 & Q
}
\]
where $q$ is a coequaliser of $f$ and $g$, and $e$ is their equaliser. Then the bottom row being an equaliser forces $q$ to be an isomorphism and hence $f = g$. Thus, the commutativity of equalisers and coequalisers in $\C$ would force $\C$ to be a preorder. Let us now consider all such commutation conditions between finite products, coproducts, equalisers and coequalisers. Representing such a condition by an edge (whose vertices are the respective limit/colimit) we have the following diagram:
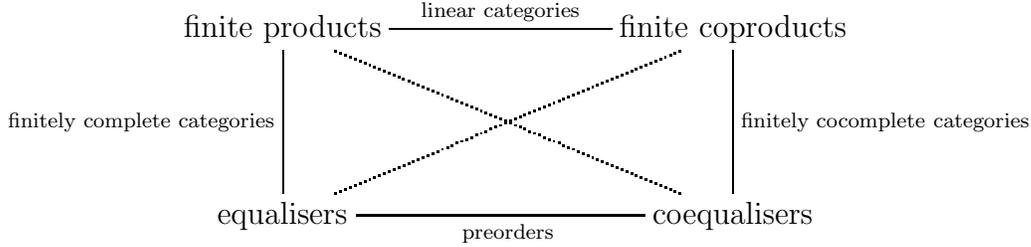
\[
\xymatrix{
\text{finite products} \ar@{..}[rrrdd] \ar@{-}[dd]_{\text{finitely complete categories}} \ar@{-}[rrr]^-{\text{linear categories}}& & & \text{finite coproducts}  \ar@{-}[dd]\ar@{-}[dd]^{\text{finitely cocomplete categories}}\\
\\
\text{equalisers}  \ar@{..}[rrruu]  \ar@{-}[rrr]_{{\text{preorders}}} & & & \text{coequalisers}
}
\]
The labels of the solid arrows refer to class of categories determined by the respective commutation condition, while the dotted diagonal lines represent conditions which are dual to each other, namely, the commutation of finite products with coequalisers and of finite coproducts with equalisers. Perhaps it is surprising how common the commutativity of finite products with aribitrary coequalisers is in algebra: the categories of groups, rings, monoids, Boolean algebras, (non-empty) implication algebras, or congruence modular/distributive varieties with constants, but also categorical examples such as linear, regular unital (or strongly unital) categories \cite{Bou96} (admitting coequalisers), regular pointed majority categories \cite{Hoe18a, Hoefnagel2020a} (admitting coequalisers), as well as pointed regular Gumm or factor permutable categories (admitting coequalisers) \cite{BournGran2004, Gra04}, all provide us with examples of categories in which finite products and coequalisers commute\footnote{General varieties of universal algebras in which finite products commute with coequalisers have be syntactically described \cite{Hoefnagel2023}. This following the corresponding characterisation for pointed varieties \cite{Hoefnagel2019b}.}. Hence the question of what generally, if anything, can be said of categories in which finite products commute with coequalisers. In this paper, we focus this question to pointed categories, and show that much of the behavior of central morphisms \cite{Huq1968} in unital categories \cite{Bou96, Bou02} extends to every pointed category $\C$ in which binary products commute with coequalisers. 

Recall that two morphisms $f\colon A \rightarrow X$ and $g\colon B \rightarrow X$ in a pointed category $\C$ with binary products are said to \emph{commute} \cite{Huq1968} (or Huq-commute) if there exists a morphism $\rho\colon A \times B \rightarrow X$ such that $\rho (1_A,0) = f$ and $\rho (0,1_B) = g$, where $(1_A,0)$ and $(0,1_B)$ are the canonical product inclusions.
\[
\xymatrix{
A \ar[rd]_{f} \ar[r]^-{(1_A,0)} & A \times B \ar@{..>}[d]_-\rho & B \ar[l]_-{(0,1_B)} \ar[ld]^{g} \\
& X &
}
\]
Such a morphism $\rho$ we call a \emph{cooperator} for $f$ and $g$, following the terminology of \cite{Bou02}.
For instance, in the category $\Grp$ of groups, two subgroups $A \hookrightarrow X$ and $B \hookrightarrow X$ of a group $X$ commute if they centralise each other, i.e., $ab = ba$ for all $a \in A$ and $b\in B$.  A morphism $f\colon  A \rightarrow B$ in $\C$ is \emph{central} when it commutes with the identity $1_B$ on $B$, and an object $A$ is called \emph{commutative} if $1_A$ is central. Thus, a subgroup $A \hookrightarrow X$ is central in $\Grp$ if and only if it is a central subgroup, i.e., it is contained in the center $\Z(X)$ of $X$. Recall from \cite{Bou96} that a finitely complete pointed category $\C$ is \emph{unital} if the morphisms $A \xrightarrow{(1_A,0)} A \times B \xleftarrow{(0,1_B)} B$ are jointly strongly epimorphic for any two objects $A$ and $B$ of $\C$. In the context of a unital category the class of central morphisms $\Z(\C)$ in $\C$ forms a right ideal, and is such that between any two objects $X,Y$ in $\C$ the class of central morphisms $\Z(X,Y)$ between $X$ and $Y$ forms a commutative monoid with a canonical action on $\C(X,Y)$. This is the so-called ``additive core'' of \cite{Bou02}. In this paper, we show that this additive core is present in any pointed category $\C$ admitting  binary products and coequalisers and which commute. More generally, this is shown for any \emph{centralic} category in the sense defined in section~\ref{section: preliminaries and examples}. In particular, every unital or weakly unital category \cite{Nelson2008} is centralic, but also contexts far from unital, such as every pointed majority category \cite{Hoe18a, Hoefnagel2020a}, every pointed Gumm \cite{BournGran2004} or factor permutable category \cite{Gra04}, all provide us with examples of centralic categories. The relation of centralic categories to the commutativity of finite products with coequalisers is that a pointed category $\C$ with coequalisers, whose regular epimorphisms are stable under finite products, is centralic if and only if finite products commute with coequalisers in $\C$. For a centralic category $\C$, the full subcategory $\Com(\C)$ of commutative objects in $\C$ is linear and equivalent to the category $\CMon(\C)$ of internal commutative monoids in $\C$. Under suitable conditions, the inclusion $\Com(\C) \rightarrow \C$ has a finite product preserving left-adjoint, making the category $\Com(\C)$ a Birkhoff subcategory of $\C$. Moreover, we show that the subcategory $\Ab(\C)$ of abelian objects in $\C$ is reflective in $\Com(\C)$. 

\subsubsection*{Convention and notation}
Throughout this paper we will write $0$ for the zero object in a given pointed category $\C$, as well as for zero-morphisms in $\C$ --- provided the context is clear, and does not lead to confusion. We will also make frequent use of the language of generalised elements in what follows, and give set-theoretic arguments in proofs which involve finite limits, understanding that these arguments generalise to categories via standard techniques involving the Yoneda embedding (see Metatheorem~0.1.3 in \cite{BorceuxBourn}). We will also freely use general properties of regular categories \cite{BGO71} as can be found in \cite{Gran2021}.

\section{Centralic categories} \label{section: preliminaries and examples}

\begin{proposition} \label{proposition: reformulations of Ax1}
The following are equivalent for a pointed category with binary products:
\begin{enumerate}[(i)]
    \item for any $f\:X \times X \rightarrow Y$ we have $f(x,0) = f(x',0)$ implies $f(x,y) = f(x',y)$,
    \item for any $f\:X \times Y \rightarrow Z$ we have $f(x,0) = f(x',0)$ implies $f(x,y) = f(x',y)$,
    \item for any $f\:X \times Y \rightarrow Z$ we have $f(x,0) = f(x',0)$ and $f(0,y) = f(0,y')$ implies $f(x,y) = f(x',y')$.
\end{enumerate}
\end{proposition}
\begin{proof}
For $(i) \implies (ii)$, consider the morphism $\alpha \: (X \times Y) \times (X \times Y) \rightarrow Z \times Z$ defined by $\alpha((x,y),(x',y')) = (f(x,y'),f(x',y))$. Then $\alpha((x,0),(0,0)) = \alpha((x',0),(0,0))$ so that by $(i)$ we get $\alpha((x,0),(0,y)) = \alpha((x',0),(0,y))$ and hence $f(x,y) = f(x',y)$. For $(ii) \implies (iii)$ we have $f(x,0) = f(x',0) \implies f(x,y) = f(x',y)$ and $f(0,y) = f(0,y') \implies f(x',y) = f(x',y')$ so that $f(x,y) = f(x',y')$. Then $(iii) \implies (i)$ is trivial.
\end{proof}
\begin{definition}
A pointed category $\C$ with finite products is called \emph{centralic} if it satisfies any one of the equivalent conditions of Proposition~\ref{proposition: reformulations of Ax1}.
\end{definition}

For a finitely complete category $\C$ we write $\Eq(f)$ for the kernel equivalence relation of a morphism $f$ in $\C$, i.e., the equivalence relation represented by the kernel pair of $f$. Then the property of a finitely complete pointed category $\C$ to be centralic may formulated with respect to effective equivalence relations: $\C$ is centralic if and only if for any effective equivalence relation $\Theta$ on a product $X \times Y$ in $\C$ we have $$(x,0)\Theta (x',0) \implies (x,y)\Theta (x',y).$$ This is a direct formulation of the diagrammatic condition of (3) of Proposition~2.9 in \cite{Hoefnagel2019b} so that the proposition below is just a reformulation of that proposition. 
\begin{proposition} \label{proposition: products commute with coequalisers iff A}
Let $\C$ be a pointed finitely complete category with coequalizers, then the following are equivalent.
\begin{enumerate}
\item Binary products commute with coequalisers in $\C$, i.e., for any two coequalizer diagrams
\[
    \xymatrix{
    C_1 \ar@<-.5ex>[r]_{v_1} \ar@<.5ex>[r]^{u_1} & X_1 \ar[r]^{q_1} & Q_1  & C_2 \ar@<-.5ex>[r]_{v_2} \ar@<.5ex>[r]^{u_2}& X_2 \ar[r]^{q_2} & Q_2,
    }
\]
in $\C$, the diagram
\[
\xymatrix{
& C_1 \times C_2 \ar@<-.5ex>[r]_-{v_1 \times v_2} \ar@<.5ex>[r]^-{u_1 \times u_2}  & X_1 \times X_2 \ar[r]^-{q_1 \times q_2} & Q_1 \times Q_2, \\
}
\]
is a coequaliser diagram.
\item For any regular epimorphism $q: X \rightarrow Y$ and any object $Z$ in $\C$, the diagram
\[
\xymatrix{
X \ar[r]^{q} \ar[d]_{(1_X, 0)} & Y \ar[d]^{(1_Y, 0)} \\
X \times Z \ar[r]_{q \times 1_Z} & Y \times Z
}
\]
is a pushout.
\item Regular epimorphisms are stable under binary products and $\C$ is centralic.
\end{enumerate}
\end{proposition}
Since regular epimorphisms are stable under binary products in any regular category \cite{BGO71}, the following corollary is immediate.
\begin{corollary}
A pointed regular category with coequalisers $\C$ is centralic if and only if binary products commute with coequalisers.
\end{corollary}
\begin{proposition} \label{proposition: commutativity prod.coeq implies centralic}
Every pointed category $\C$ admitting binary products and coequalisers which commute is centralic.
\end{proposition}
\begin{proof}
Suppose that $f\:X \times Y \rightarrow Z$ is any morphism in $\C$. Given $x,x'\:S \rightarrow X$ such that $f(x,0) = f(x',0)$, consider a coequaliser diagram
\[
\xymatrix{
S \ar@<-.5ex>[r]_-{x} \ar@<.5ex>[r]^-{x'} & X \ar[r]^q & Q.
}
\]
Since $0 \rightrightarrows Y \xrightarrow{1_Y} Y$ is trivially a coequaliser, the diagram
\[
\xymatrix{
S  \ar@<-.5ex>[r]_-{(x,0)} \ar@<.5ex>[r]^-{(x',0)} & X \times Y  \ar[r]^{q \times 1_Y} & Q \times Y
}
\]
is a coequaliser, which implies that $f(x,y) = f(x',y)$ for any $y\:S \rightarrow Y$.
\end{proof}
\begin{proposition} \label{proposition: kernel cooperator}
If $\C$ is centralic and has kernel pairs, given morphisms $f\:X \rightarrow Y$ and $g\:X' \rightarrow Y$ which admit a cooperator $\rho\:X \times X' \rightarrow Y$, we have $\Eq(f\times g) \leqslant \Eq(\rho)$.
\end{proposition}
\begin{proof}
If $(f\times g)(x,y) = (f\times g)(x',y')$ then $\rho(x,0) = \rho(x',0)$ and $\rho(0,y) = \rho(0,y')$, so that by Proposition~\ref{proposition: reformulations of Ax1} $(iii)$ we get $\rho(x,y) = \rho(x',y')$, and the result follows.
\end{proof}
\subsection{Examples of centralic categories} \label{section: examples}
A pointed finitely complete category $\C$ is \emph{weakly unital} \cite{Nelson2008} if for every two objects $X$ and $Y$ the product inclusions $$X \xrightarrow{(1_X,0)} X \times Y \xleftarrow{(0,1_Y)} Y$$ of any binary product diagram are jointly epimorphic. As an example, consider any pointed quasi-variety $\V$ of algebras which admits a binary operation $+$ satisfying $x + 0 = 0 + x$ and $x + 0 = y + 0 \implies x = y$. According to Proposition~3.2 in \cite{Gray2012} we have that $\V$ is weakly unital.
\begin{proposition} \label{proposition: example-unital}
Every weakly unital category is centralic.
\end{proposition}
\begin{proof}
Given $x,x'\:S \rightarrow X$, $y\:S \rightarrow Y$, and any morphism $f\:X \times Y \rightarrow Z$ with $f(x,0) = f(x',0)$ the morphisms $f(x \times y)$ and $f   (x' \times y)$ are equal when composed with the canonical product inclusions $S \xrightarrow{(1_S,0)} S \times S \xleftarrow{(0,1_S)} S$. Thus they are equal, and the result follows.
\end{proof}
The notion of a \emph{Gumm category} \cite{BournGran2004} is the categorical analogue of varieties in which Gumm's shifting lemma holds \cite{Gum83}, i.e., congruence modular varieties. A finitely complete category $\C$ is a Gumm category if for any three equivalence relations $R,S,T$ on any object $X$ in $\C$ such that $R \cap S \leqslant T$, if $(x,y),(w,z) \in R$ and $(y,z), (x,w) \in S$ and $(y,z) \in T$ then we get $(x,w) \in T$. This implication of relations between the elements above is usually depicted with a diagram
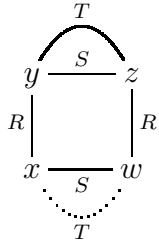
\[
\xymatrix{
y \ar@/^1.5pc/@{-}[r]^{T} \ar@{-}[r]^{S} \ar@{-}[d]_{R} & z \ar@{-}[d]^{R} \\
x\ar@/_1.5pc/@{..}[r]_{T} \ar@{-}[r]_{S} & w
}
\]
\begin{proposition} \label{proposition: example pointed gumm}
Every pointed Gumm category is centralic.
\end{proposition}
\begin{proof}
Given a morphism $f\:X \times Y \rightarrow Z$ such that $f(x,0) = f(x',0)$ then we have the diagram below where $\Eq(\pi_1) \cap \Eq(\pi_2) \leqslant \Eq(f)$ holds trivially.
\[
\xymatrix{
(x,0) \ar@/^1.5pc/@{-}[r]^{\Eq(f)} \ar@{-}[r]^{\Eq(\pi_2)} \ar@{-}[d]_{\Eq(\pi_1)} & (x',0) \ar@{-}[d]^{\Eq(\pi_1)} \\
(x,y) \ar@/_1.5pc/@{..}[r]_{\Eq(f)} \ar@{-}[r]_{\Eq(\pi_2)} & (x',y)
}
\]
\end{proof}
A regular category $\C$ is said to be \emph{factor permutable} \cite{Gra04} if for every object $A$ in $\C$ and any equivalence relation $E$ on $A$ we have that $\Eq(p) \circ E = E \circ \Eq(p)$ for every product projection $p\:A \rightarrow X$ of $A$. This notion is the categorical generalisation of factor permutable varieties introduced in \cite{Gum83}.
\begin{remark} \label{remark: factor permutable}
Applying  Lemma~2.5 (the weak shifting lemma) in \cite{Gra04} to the same diagram as in Proposition~\ref{proposition: example pointed gumm} it will follow that every pointed factor permutable variety is centralic. 
\end{remark}
\begin{remark}
Proposition~\ref{proposition: example pointed gumm} may also be seen as a consequence of the fact that any punctually congruence hyperextensible category in the sense of \cite{Bourn2005} is centralic. 
\end{remark}
The notion of a \emph{majority category} has been defined in the paper \cite{Hoe18a} (see also \cite{Hoefnagel2020a}). We now recall this notion for the reader's convenience. Consider the condition on a ternary relation $R$ between sets $X,Y,Z$ given by
	\[
	(x,y,z') \in R \quad \text{and} \quad (x,y',z) \in R \quad  \text{and} \quad (x',y,z) \in R  \implies (x,y,z) \in R.
	\]
	Then a category $\C$ is a \emph{majority category} if every internal ternary relation in $\C$ satisfies the above condition internalised (via the Yoneda embedding) to internal ternary relations in $\C$ (see \cite{Hoe18a} for the details). Majority categories capture, in a categorical way, what it means for a variety of universal algebras to admit a \emph{majority} term, i.e., a ternary term $m(x,y,z)$ satisfying the equations
 \begin{align*}
     m(x,x,y) = x, \\
     m(x,y,x) = x, \\
     m(y,x,x) = x.
 \end{align*}
 For instance, in the variety $\mathsf{Lat}$ of lattices, the term $m(x,y,z) = (x \wedge y) \vee (x \wedge z)  \vee (y \wedge z)$ or its lattice-theoretic dual both define majority terms. In a variety of rings satisfying the identity $x^n = x$ for some $n \geqslant 2$ (such as a variety of rings generated by a finite field for example), then the term $m(x,y,z) = x - (x - y)(x - z)^{n-1}$ defines a majority term. 
\begin{proposition} \label{proposition: example majority category}
Every finitely complete pointed majority category is centralic.
\end{proposition}
\begin{proof}
Given a morphism $f\:X \times Y \rightarrow Z$ with $f(x,0) = f(x',0)$ we define the ternary relation $R$ between $X$ and $X$ and $Y\times Y$
by
\[
(x,x',(y,y'))\in R \Longleftrightarrow f(x,y) = f(x',y').
\]
Then  we have
\[
(x,x',(0,0)) \in R, \quad (x,x,(y,y)) \in R, \quad (x',x',(y,y)) \in R \implies   (x,x',(y,y)) \in R,
\]
and hence $f(x,y) = f(x', y)$.
\end{proof}

\section{Centrality in centralic categories}
Throughout this section we fix a centralic category $\C$. 
\begin{proposition} \label{proposition: uniqueness of cooperator}
Given any central morphism $f\:X\rightarrow Y$ in $\C$ and any morphism $g\:X'\rightarrow Y$ in $\C$, there exists a unique cooperator $\rho_{f,g}\:X \times X' \rightarrow Y$ for $f$ and $g$. This cooperator is defined by $\rho_{f,g}(x,y) = \rho_f(x,g(y))$ where $\rho_f \:X\times Y \rightarrow Y$ is the (unique) cooperator for $f$ and $1_Y$.
\end{proposition}
\begin{proof}
Let  $\rho_f \:X\times Y \rightarrow Y$ be a cooperator for $f$ and $1_Y$, then it is easily checked that the morphism $\rho_{f,g}\:X \times X' \rightarrow Y$ defined by $\rho_{f,g}(x,y) = \rho_f(x,g(y))$ is a cooperator for $f$ and $g$. For uniqueness, suppose that $\rho:X \times X' \rightarrow Y$ is any cooperator for $f$ and $g$. Consider the morphism $\alpha\:(X \times X) \times X' \rightarrow Y$ defined by
\[
\alpha((x,y),z) = \rho_f(x, \rho(y,z)).
\]
Then $\alpha((x,0), 0) = \alpha((0,x), 0)$ since
\[
\alpha((x,0), 0) = \rho_f(x, \rho(0,0)) = \rho_f(x,0) = f(x) = \rho_f(0, f(x)) = \rho_f(0, \rho(x, 0) =  \alpha((0,x), 0),
\]
so that since $\C$ is centralic, we have $\alpha((x,0), y) = \alpha((0,x), y)$, and hence
\begin{align*}
\alpha((x,0), y) &= \alpha((0,x), y) \implies \\
 \rho_f(x, \rho(0,y)) &= \rho_f(0, \rho(x,y)) \implies \\
\rho_f(x,g(y)) &= \rho(x, y),
\end{align*}
which gives us uniqueness. 
\end{proof}
\begin{remark}
A natural question in light of the proposition above is if every finitely complete centralic category has that any cooperator between two  morphisms with the same codomain is unique, i.e., is every centralic category weakly unital? We can answer this question in the negative by considering Proposition~\ref{proposition: example majority category}, which states that every finitely complete pointed majority category is centralic. But not every pointed majority category is weakly unital: if $\C$ is a majority category then every category of points $\Pt_{\C}(X)$ (see \cite{Bou96} for an exact definition) is a majority category (see Example~2.15 in \cite{Hoe18a}). Thus in the variety $\mathsf{Lat}$ of lattices (the prototypical example of majority category) every category of points $\Pt_{\mathsf{Lat}}(X)$ (where $X$ is a lattice) a finitely complete majority category. However not every category of points $\Pt_{\mathsf{Lat}}(X)$ is weakly unital, since if it were then $\mathsf{Lat}$ would be weakly Mal'tsev in the sense of \cite{Nelson2008, Ferreira2008}, and a variety of lattices is weakly Mal'tsev if and only if it is a variety of distributive lattices \cite{Ferreira2012, Ferreira2015}. 
\end{remark}
The corollary below is the analogue of Theorem~3.1.13 in \cite{Huq1968} and Theorem~4.2 in \cite{Bou02}, in our context, and its proof is similar.
\begin{corollary} \label{corollary: cooperator central}
Given two central morphisms $f\:X \rightarrow Y$ and $g\:X'\rightarrow Y$ in $\C$ their cooperator $\rho_{f,g}\:X \times X' \rightarrow Y$ is central.
\end{corollary}
\begin{proof}
Let $\rho_f\:X \times Y \rightarrow Y$ be the cooperator for $f$ and let $\rho_g\:X' \times Y \rightarrow Y$ be the cooperator for $g$ and $1_Y$. By Proposition~\ref{proposition: uniqueness of cooperator} we have that the cooperator $\rho$ for $f$ and $g$ is precisely given by $\rho = \rho_f   (1 \times g)$. It is then easily checked that the morphism $\alpha\:(X \times X') \times Y \rightarrow Y$ defined by $\alpha((x,x'), y) = \rho_f(x ,\rho_g(x',y))$ is a cooperator for $\rho$.
\end{proof}
Following \cite{Bou96} let us write $\Z(\C)$ for the class of all central morphisms in $\C$. Given objects $X$ and $Y$ in $\C$, we write $\Z(X,Y)$ for all central morphisms from $X$ to $Y$ (as in \cite{Bou96}). The below proposition is similar to Theorem~4.1 in \cite{Bou96} and the proof is the same.
\begin{proposition} \label{proposition: Z right ideal}
The class $\Z(\C)$ is a right-ideal of $\C$.
\end{proposition}
\begin{proof}
If $f\:X\rightarrow Y$ is central with cooperator $\rho_f$ and $x\:S \rightarrow X$ is any morphism then $\rho_f   (x \times 1_Y)$ makes $f  x$ central. 
\end{proof}
Note that, according to the proposition above, the class $\Z(\C)$ is closed under composition.
\subsection{The additive core of $\C$} 
When $\C$ is a unital category there is a canonical commutative monoid structure on $\Z(X,Y)$ which acts canonically on $\C(X,Y)$. This is the so-called ``additive core'' referred to in \cite{Bou02}. We will now show that this additive core is present inside any centralic category $\C$. Consider the map (as it is defined in \cite{Bou02})
\[
\Z(X,Y) \times \C(X,Y) \rightarrow \C(X,Y), \quad (f,g) \longmapsto f \star g = \rho_{f,g}(1_X,1_X),
\]
where $\rho_{f,g}$ is the unique cooperator (Proposition~\ref{proposition: uniqueness of cooperator}) of $f$ and $g$. Note that we have
\[
\rho_{f,g} = \rho_f (1_X \times g)
\]
as the unique cooperator of $f$ and $g$. The zero morphism $0\:X \rightarrow Y$ is central with cooperator $\pi_2\:X \times Y \rightarrow Y$, so that for any $g\in \C(X,Y)$ we have by the above formulas that $0\star g = \pi_2 (1_X \times g) (1_X,1_X) = g$. Now suppose that $f,g \in \Z(X,Y)$ and $h\in\C(X,Y)$ and write $\rho_{f,g}$ for the unique cooperator between $f$ and $g$ (Proposition~\ref{proposition: uniqueness of cooperator}), and write $\rho_{(f,g), h}$ for the unique cooperator between $\rho_{f,g}$  and $h$ (since $\rho_{f,g}$ is central by Corollary~\ref{corollary: cooperator central}). Similarly, $\rho_{f, (g,h)}$ is the unique cooperator between $f$ and $\rho_{g,h}$. The (necessarily unique) cooperator for $f \star g$ and $h$ is just $\rho _{(f,g),h}[(1_X,1_X)\times 1_X]$ since we have the commutative diagram
\[
\xymatrix{
X \ar[d]_{(1_X,1_X)} \ar[rr]^-{(1_X,0)} & & X \times X \ar[d]_{(1_X,1_X)\times 1_X} & & X \ar[ll]_-{(0,1_X)} \ar[d]^{1_X} \\
X \times X \ar[rr]^-{(1_{X \times X},0)} \ar[rrd]_{\rho_{f,g}} & & (X \times X) \times X \ar[d]_{\rho_{(f,g), h}} & & X \ar[ll]_-{(0,1_X)} \ar[lld]^h \\
& & Y & 
}
\]
so that
\[
(f \star g) \star h = \rho_{(f,g),h}((1_X,1_X),1_X).
\]
Similarly, we also have
\[
f \star (g \star h) =  \rho_{f,(g,h)}(1_X, (1_X,1_X)).
\]
Let $\alpha\:X \times (X \times X) \rightarrow (X \times X) \times X$ be the canonical associativity isomorphism. It is then routine to check that the composite $\rho_{(f,g),h} \alpha (0,1_{X \times X})$ defines a cooperator for $g$ and $h$  and is therefore equal to $\rho_{g,h}$ by Proposition~\ref{proposition: uniqueness of cooperator}. Similarly, it may be checked that $\rho_{(f,g),h} \alpha (1_X,(0,0)) = f$, so that the morphism $\rho_{(f,g),h} \alpha$ defines a cooperator for $f$ and $\rho_{g,h}$ and hence  $\rho_{(f,g),h} \alpha= \rho_{f,(g,h)}$ --- by Proposition~\ref{proposition: uniqueness of cooperator}. Then we have: 
\begin{align*}
f \star (g \star h) &= \rho_{f,(g,h)}(1_X,(1_X,1_X)) \\
&= \rho_{(f,g),h} \alpha (1_X, (1_X, 1_X)) \\
&= \rho_{(f,g),h} ((1_X, 1_X), 1_X)) \\
&= (f \star g) \star h.
\end{align*}
Note that by Corollary~\ref{corollary: cooperator central} and Proposition~\ref{proposition: Z right ideal} if $f, g\in \Z(X,Y)$ then $f \star g \in \Z(X,Y)$ so that by the above remarks the map $\star$ above restricts to a monoid operation on $\Z(X,Y)$ which we will denote by $+$. This monoid is commutative since if $f,g \in \Z(X,Y)$ then their cooperator $\rho_{f,g}$ is explicitly given by
\[
\rho_g   (1_X \times f) = \rho_{f,g} = \rho_f   (1_X \times g).
\]
Then the map $\star$ given above becomes a commutative monoid action on $\C(X,Y)$. Given any morphism $x\:X' \rightarrow X$ the canonical map $\C(x,Y)\:\C(X,Y) \rightarrow \C(X',Y)$ is a homomorphism of monoid actions, and the proof of Proposition~4.4 in \cite{Bou02} applies directly: given $f \in \Z(X,Y)$ and $g \in \C(X,Y)$ we have
\begin{align*}
(f \star g) x &= \rho_{f,g}(1_X,1_X)x \\
            &= \rho_f(1_X \times g) (x \times x )(1_{X'},1_{X'}) \\
            &= \rho_f(x \times gx)(1_{X'},1_{X'}) \\
            &= \rho_{fx}(1_{X'} \times gx) (1_{X'},1_{X'}) \\
            &= fx \star gx
\end{align*}
Moreover, restricting the canonical map $\C(x,Y)$ to $\Z(X,Y)$ produces a monoid homomorphism $\Z(x,Y)\:\Z(X,Y) \rightarrow \Z(X',Y)$. 

\subsubsection*{Symmetrizable morphisms}
Recall from \cite{Bou02} the notion of a \emph{symmetrizable morphism} (Definition~4.2 in \cite{Bou02}):
\begin{definition}
A morphism $f\:X \rightarrow Y$ in $\C$ is symmetrizable if it has an inverse in $\Z(X,Y)$. A commutative object $X$ is called abelian if $1_X$ is symmetrizable. The full subcategory of abelian objects in $\C$ is denoted by $\Ab(\C)$.
\end{definition}
\noindent
As in \cite{Bou02}, for two objects $X,Y$ in $\C$ we write $\ZZ(X,Y)$ for all symmetrizable morphisms from $X$ to $Y$. The same proofs of Proposition~4.7, 4.8, Theorem~4.3, Proposition~4.6 in \cite{Bou02} apply so that we have
\begin{itemize}
    \item A central morphism $f\:X \rightarrow Y$ is symmetrizable if and only if
    \[
    \xymatrix{
    X \times Y \ar[d]_{\pi_1} \ar[r]^-{\rho_f} & Y \ar[d] \\
    X \ar[r]& 0
    }
    \]
    is a pullback;
    \item if $f\:X \rightarrow Y$ and $g\:X' \rightarrow Y$ are symmetrizable then their cooperator $\rho_{f,g}\:X \times X' \rightarrow Y$ is symmetrizable;
    \item the class $\ZZ(\C)$ is a right ideal and $\ZZ(X,Y)$ is an abelian group which has a canonical action on $\C(X,Y)$;
    \item an object $X$ is abelian if and only if is the underlying object of a internal abelian group (which is necessarily unique).
\end{itemize}
\begin{proposition} \label{proposition: qoutient commutative}
Let $\C$ be a finitely complete centralic category and suppose that regular epimorphisms are stable under binary products in $\C$. In the diagram
\[
\xymatrix{
A \ar[d]_{q_1} \ar[r]^-{(1_A,0)} & A \times B \ar[d]| {q_1 \times q_2} & B \ar[d]^{q_2} \ar[l]_-{(0,1_B)} \\
X \ar@/_/[rd]_{f}  \ar[r]^-{(1_X,0)} & X \times Y & Y \ar[l]_-{(0,1_Y)} \ar@/^/[ld]^g \\
& Z &
}
\]
where $q_1$ and $q_2$ are regular epimorphisms, we have that $f$ commutes with $g$ if and only if $fq_1$ commutes with $gq_2$.
\end{proposition}
\begin{proof}
Suppose $f$ and $g$ commute with cooperator $\rho$, then $\rho (q_1 \times q_2)$ is a cooperator for $fq_1$ and $gq_2$. Conversely, suppose that $\rho\:A \times B \rightarrow Z$ is a cooperator for  $fq_1$ and $gq_2$. By Proposition~\ref{proposition: kernel cooperator} we have $\Eq\big ((fq_1) \times (gq_2)\big) \leqslant \Eq(\rho)$, and since we always have $\Eq(q_1 \times q_2) \leqslant \Eq\big ((fq_1) \times (gq_2)\big)$ we get $\Eq(q_1 \times q_2) \leqslant \Eq(\rho)$. Since $q_1 \times q_2$ is a regular epimorphism the dotted arrow exists making the triangle
\[
\xymatrix{
A \times B \ar[rd]_-{\rho} \ar[r]^-{q_1 \times q_2} & X \times Y \ar@{..>}[d]^-{\rho'} \\
& Z
}
\]
commute. To see that $\rho'$ is a cooperator for $f$ and $g$, we have
\[
\rho'(1_{X},0)q_1 = \rho' (q_1 \times q_2) (1_A,0) = \rho (1_A,0) = f q_1
\]
and since $q_1$ is an epimorphism we have $\rho'(1_{X},0) = f$. The equality $\rho'(0,1_{Y}) = g$ follows similarly. 
\end{proof}
Given any finitely complete centralic category $\C$ whose regular epimorphisms are stable binary products, we have the following two corollaries:
\begin{corollary} \label{corollary: X commutative if q1 commutes q2} 
For any two regular epimorphisms $q_1\:A \rightarrow X$ and $q_2\:B \rightarrow X$ in $\C$, the object $X$ is commutative if and only if $q_1$ commutes with $q_2$.
\end{corollary}

\begin{corollary} \label{corollary: qoutient commutative/abelian}
 Let $q\:A \rightarrow A'$ be a regular epimorphism in $\C$ then for any morphism $f\:A' \rightarrow Z$ in $\C$ if $fq$ is central then so is $f$. If $\C$ is regular, then $fq$ symmetrizable implies $f$ symmetrizable.
\end{corollary}
\begin{proof}
If $fq$ is central then we apply Proposition~\ref{proposition: qoutient commutative} with $q_1 = q$ and $q_2 = g = 1_Z$, so that $f$ is central. If $fq$ is symmetrizable, then the same argument in the proof of Proposition~4.10 in \cite{Bou02} works in our situation: suppose that $\rho_f$ is the cooperator for $f$ and $1_Z$ and consider the diagram below.
\[
\xymatrix{
A \times Z \ar[d]_{\pi_1} \ar[r]^{q \times 1_Z} & A' \times Z \ar[d]_{\pi_1} \ar[r]^-{\rho_f} & Z \ar[d] \\
A \ar[r]_q & A' \ar[r] & 0
}
\]
The outer rectangle is a pullback by assumption, since $\rho_f (q \times 1_Z) = \rho_{f, q}.$ The left-hand bottom horizontal morphism is a regular epimorphism, consequently the right-hand square is a pullback since $\C$ is a regular category (see Lemma 1.15 in \cite{Gran2021}, for instance). 
\end{proof}
The proposition below is similar to Proposition~4.9 in \cite{Bou02}.
\begin{proposition} \label{proposition: central qoutient is central}
Let $\C$ be a finitely complete centralic category and suppose that regular epimorphisms in $\C$ are stable under binary products. Given that $f\:A\rightarrow Y$ is central/symmetrizable in $\C$ and $q\:Y \rightarrow Q$ is a regular epimorphism  in $\C$, then $qf$ is central/symmetrizable.
\end{proposition}
\begin{proof}
Let $\rho_f$ be the cooperator of $f$ and $1_Y$, then the morphism $1_A \times q\:A \times Y \rightarrow A \times Q$ is a regular epimorphism such that $\Eq(1_A \times q) \leqslant \Eq(q \rho_f)$. Indeed, one has the following string of implications
\begin{align*}
(1_X \times q)(x,y) &= (1_X \times q)(x,y') \implies \\
q(y) &= q(y') \implies \\
q\rho_f (0, y) &= q\rho_f (0, y') \implies \\
q\rho_f (x, y) &= q\rho_f (x, y'), 
\end{align*}
so that there exists a morphism $\rho\:A \times Q \rightarrow Q$ such that $q\rho_f = \rho (1_A \times q)$. Then $\rho$ is the desired cooperator for $qf$ and $1_Q$, since $$\rho (1_A,0) = \rho(1_A \times q)(1_A,0) =  q \rho_f (1_A,0) = f$$ and $$\rho(0,1_Q)q = \rho (1_A \times q) (0,1_Y) = q \rho_f (0,1_Y) = q 1_Y = 1_Q q$$ which implies $\rho(0,1_Q) = 1_Q$ since $q$ is a (regular) epimorphism. In the case that $f$ is symmetrizable, the same argument in Proposition~4.9 in \cite{Bou02} may be applied to show that $qf$ is symmetrizable.
\end{proof}
The following corollary is immediate.
\begin{corollary} \label{corollary: commutative closed under qoutients}
If regular epimorphisms are stable under binary products in $\C$ then the subcategories $\Com(\C)$ and $\Ab(\C)$ are closed under regular quotients in $\C$, i.e., for any regular epimorphism $q\:X \rightarrow Y$ if $X$ is commutative/abelian then so is $Y$.
\end{corollary}

\subsection{Commutative objects}
Recall that an internal \emph{unitary magma} in $\C$ is a pair $(X,\rho_X)$ making the diagram below commute.
\[
\xymatrix{
X \ar[rd]_{1_X} \ar[r]^-{(1_X,0)} & X \times X \ar[d]^{\rho_X} & \ar[l]_-{(0,1_X)} X \ar[ld]^{1_X} \\
& X
}
\]
Therefore an object is commutative if and only if it has a (necessarily unique) unitary magma structure. A morphism $f\:X \rightarrow Y$ in $\C$ where $(X, \rho_X)$ and $(Y, \rho_Y)$ are internal unitary magmas is a homorphism of internal unitary magmas if the diagram below commutes.
\[
\xymatrix{
X \times X \ar[d]_{\rho_X} \ar[r]^{f \times f} & Y \times Y \ar[d]^{\rho_Y} \\
X \ar[r]_f & Y
}
\]
\begin{proposition} \label{proposition: morphism between commutative objects}
Given any morphism $f\:X \rightarrow Y$ in $\C$ where $(X,\rho_X)$ and $(Y,\rho_Y)$ are unitary magmas, then $f$ is a morphism of unitary magmas. 
\end{proposition}
\begin{proof}
 The morphism $f$ is central (since $1_Y$ is) and the morphisms $f\rho_X$ and $\rho_Y (f \times f)$ define cooperators for $f$ with itself, so that $f\rho_X = \rho_Y (f \times f)$ by Proposition~\ref{proposition: uniqueness of cooperator}.
\end{proof}
\begin{proposition} \label{proposition: commutative object monoid}
Every commutative object in $\C$ is the underlying object of an internal commutative monoid in $\C$ (which is necessarily unique).
\end{proposition}
\begin{proof}
By the classical Eckmann-Hilton argument \cite{EckmannHilton1962} a unitary magma $(X, \rho_X)$ is a commutative monoid if and only if its multiplication $\rho_X$ is a morphism of the unitary magmas $(X\times X, (\rho_X \times \rho_X)m)$ and $(X, \rho_X)$, where $m\:(X\times X)^2 \rightarrow (X\times X)^2$ is the middle interchange isomorphism $m\:((x,y),(a,b)) \mapsto ((x,a), (y,b))$. Thus the result follows by Proposition~\ref{proposition: morphism between commutative objects}. 
\end{proof}
Then Proposition~\ref{proposition: morphism between commutative objects} and Proposition~\ref{proposition: commutative object monoid} give us the following
\begin{corollary} \label{corollary: commutative objects are full}
The full subcategory $\Com(\C)$ of commutative objects in $\C$ is equivalent to the category $\CMon(\C)$ of internal commutative monoids in $\C$.
\end{corollary}

\begin{proposition} \label{proposition: commuative object coequaliser}
For any commutative object $X$ the diagram
\[
\xymatrix{
X  \ar@<-.5ex>[r]_-{(0,1_X)} \ar@<.5ex>[r]^-{(1_X,0)}& X \times X \ar[r]^-{\rho_X} & X
}
\]
is a coequaliser in $\C$.
\end{proposition}
\begin{proof}
Let $f\: X \times X \rightarrow Y$ be any morphism such that $f(1_X,0) = f(0,1_X)$ then it is enough to show that the triangle in the diagram below is commutative for the existence part of the statement, where uniqueness follows since $\rho_X$ is a split epimorphism.
\[
\xymatrix{
X \ar@<-.5ex>[r]_-{(0,1_X)} \ar@<.5ex>[r]^-{(1_X,0)} & X  \times X  \ar[r]^-{\rho_X} \ar[dr]_{f} & X  \ar[d]^{f(1_X,0)} \\
& & Y
}
\]
Consider the morphism $\alpha\:(X \times X) \times X$ defined by $\alpha((x,y), z) = f(\rho_X(x,z), y)$ then we have
\[
\alpha((y,0),0) = f(y,0) = f(0,y) = \alpha((0,y),0)
\]
so that since $\C$ is centralic we have
\[
\alpha((y,0),x) = \alpha((0,y),x) \implies f(\rho_X(y, x), 0) = f(\rho_X(x,y),0) = f(x, y).
\]
\end{proof}
\subsection{Strongly centralic categories} \label{section: strongly centralic}
In \cite{GrayPhD} the following condition was defined for a pointed category $\C$ with binary products (see condition~1.1.9 in \cite{GrayPhD}).
\begin{description}
\item [$\S$] for any object $X$ and every commutative diagram
\[
\xymatrix{
X \ar@/_/[ddr]_f \ar[r]^-{(1_X,0)}& X \times X \ar[d]_{\phi} & X \ar[l]_-{(0,1_X)} \ar@/^/[ldd]^f\\
& Y \\
& W \ar[u]^m
}
\]
where $m$ is a monomorphism, there exists a morphism $\psi\:X \times X \rightarrow W$ such that $m \psi = \phi$.
\end{description}
For example every unital category satisfies $\S$ --- see Proposition~1.1.10 in \cite{GrayPhD}. As we will shortly see, there are also non-unital examples of categories satisfying $\S$.
\begin{definition}
A pointed category $\C$ is called \emph{strongly centralic} if it is centralic and satisfies the property $\S$.
\end{definition}
\begin{proposition}
Given a strongly centralic category $\C$ the subcategory $\Com(\C)$ is closed under subobjects in $\C$.
\end{proposition}
\begin{proof}
Given a monomorphism $m:A \rightarrow X$ where $X$ is commutative, we may consider the diagram
\[
\xymatrix{
A \ar[r]^{(1_A,0)}\ar@/_2pc/[dddr]_{1_A} \ar@/_/[ddr]_m & A \times A \ar[d]^{m \times m} & A \ar@/^/[ddl]^m \ar@/^2pc/[dddl]^{1_A}  \ar[l]_{(0,1_A)} \\
 & X \times X \ar[d]|-{\rho_X} &  \\
& X \\
& A \ar[u]^m
}
\]
then condition $\S$ gives rise to the cooperator $\rho_A$ of $1_A$ with itself.
\end{proof}
Consider the following condition on pointed category $\C$ with binary products.
\begin{description}
\item[$\T$] for every object $X$ in $\C$ and any commutative diagram
\[
\xymatrix{
X \ar@/_/[dr]_f \ar[r]^-{(1_X,0)}& X \times X \ar@{->>}[d]_-{\phi} & X \ar[l]_-{(0,1_X)} \ar@/^/[ld]^f\\
& Y \\
}
\]
if $\phi$ is a regular epimorphism, then so is $f$.
\end{description}
\begin{proposition}
If $\C$ is a pointed category with binary products and coequalisers then $\T$ implies $\S$. If $\C$ has a (regular epimorphism, monomorphism)-factorisation system, then $\S$ implies $\T$.
\end{proposition}
\begin{proof}
For $\T$ implies $\S$, consider the morphisms in the diagram of $\S$ and let $q:X \times X \rightarrow Q$ be a coequaliser of $(1_X,0)$ and $(0,1_X)$. Then $q(1_X,0)$ is a regular epimorphism by $\T$, and there exists a morphism $\alpha\: Q \rightarrow Y$ such that $mf = \alpha q$. Then the dotted arrow $\theta$ in the diagram below exists since regular epimorphisms are orthogonal to monomorphisms.
\[
\xymatrix{
X \ar@{->>}[d]_{q(1_X,0)} \ar[r]^f & W \ar[d]^m \\
Q \ar@{..>}[ru]^{\theta} \ar[r]_{\alpha} & Y
}
\]
Then defining $\psi = \theta q$ makes the condition $\S$ hold. For $\S$ implies $\T$, consider the morphisms in the diagram of $\T$ and let $me = f$ be a (regular epimorphism, monomorphism)-factorisation of $f$. Then by $\S$ there exists $\psi\:X \times X \rightarrow Y$ such that $m\psi = e$ and since $e$ is regular we have that $m$ is a strong epimorphism and is therefore an isomorphism.
\end{proof}
\subsubsection*{Commutativity of binary products and coequalisers}
Throughout this section, we suppose that $\C$ is a category with binary products and coequalisers  which commute. Further, we will suppose that $\C$ satisfies the condition $\T$. By Proposition~\ref{proposition: commutativity prod.coeq implies centralic} we have that $\C$ is strongly centralic. For each object $X$ in $\C$  consider a coequaliser $X \xrightarrow{\rho_X} Q_X$ of $(1_X,0)$ and $(0,1_X)$ which gives us the diagram:
\[
\xymatrix{
X \ar[rd]_-{q_X} \ar[r]^-{(1_X,0)} & X \times X \ar[d]^-{\rho_X} &  X \ar[l]_-{(0,1_X)} \ar[ld]^-{q_X} \\
& Q_X
}
\]
Then $\T$ gives us that $q_X$ is a regular epimorphism, and hence by Corollary~\ref{corollary: X commutative if q1 commutes q2} $Q_X$ is commutative. Then we can define a left adjoint $r\:\C \rightarrow \Com(\C)$ to the inclusion $\iota\:\Com(\C) \rightarrow \C$ by defining $r(X) = Q_X$ so that $q_X$ is a universal arrow from $\iota$ to $X$. To see this, suppose that we are given any morphism $f\:X \rightarrow A$ (where $A$ is commutative and $\rho_A$ the cooperator for $1_A$) then $\rho_A(f \times f)$ coequalises $(1_X,0)$ and $(0,1_X)$ so that the necessary $f'\:r(X) \rightarrow A$ exists. Moreover, since binary products commute with coequalisers in $\C$ we can see that the reflection $r$ preserves binary products as a consequence of the diagram
\[
\xymatrix{
X \times Y \ar@<-.5ex>[rrr]_-{(0,1_X) \times (0,1_Y)} \ar@<.5ex>[rrr]^-{(1_X,0)\times (1_Y,0)} & & & (X \times X) \times (Y \times Y) \ar[r]^-{q_X \times q_Y}  & r(X) \times r(Y)
}
\]
being a coequaliser in $\C$. Moreover by Proposition~\ref{proposition: qoutient commutative} the inclusion $\iota\:\Com(\C) \rightarrow \C$ preserves coequalisers. The proposition below summarises these remarks. 
\begin{proposition}
The subcategory $\Com(\C)$ is a Birkhoff subcategory of $\C$, i.e., it is a reflective subcategory of $\C$ which is closed under products, subobjects and regular quotients. Moreover the reflection $r\:\C \rightarrow \Com(\C)$ preserves binary products, and the inclusion $\iota\:\Com(\C) \rightarrow \C$ preserves coequalisers. 
\end{proposition}

We also have an inclusion $\Ab(\C) \hookrightarrow \Com(\C)$ of the full subcategory of abelian objects in $\C$ to the full subcategory of commutative objects. The left adjoint $r\:\Com(\C) \rightarrow \Ab(\C)$ may be defined through the cokernel diagram
\[
X \xrightarrow{\Delta_X} X \times X \xrightarrow{q} r(X).
\]
where the unit $\eta$ of the adjunction has $X$ component $\eta_X = q(1_X,0)$.
\begin{proposition}
 The category $\Ab(\C)$ is a reflective subcategory of $\Com(\C)$.
\end{proposition}
\begin{proof}
Suppose that $X$ is commutative and let $q\:X \times X \rightarrow Q$ be a cokernel of the diagonal $\Delta_X\:X \rightarrow X \times X$ then $Q$ is commutative by Corollary~\ref{corollary: commutative closed under qoutients}. Then $q$ is central, and  if $\rho_{X \times X}$ and $\rho_Q$ are the additions of the respective internal commutative monoid structures on $X \times X$ and $Q$, then $q$ is a morphism of internal commutative monoids (Proposition~\ref{proposition: morphism between commutative objects}). Let $i\:X\times X \rightarrow X$ be the interchange isomorphism $(x,y) \mapsto (y,x)$ it is then easy to see that $q$ is symmetrizable with inverse $qi$, so that $Q$ is abelian by Corollary~\ref{corollary: qoutient commutative/abelian}. Then we define $r\:\Com(\C) \rightarrow \Ab(\C)$ as $r(X) = Q$ and let $\eta_X\:X \rightarrow r(X)$ be the composite $q(1_X,0)$, then $\eta_X$ is routintely checked to be universal from $\iota:\Com(\C) \rightarrow \Ab(\C)$.
\end{proof}
Every unital category is strongly centralic (see condition~1.1.9 in \cite{GrayPhD}), so that the above results apply to regular unital categories with coequalisers. Interestingly, there are categorical contexts far outside of the unital setting which are strongly centralic. For instance, consider a pointed variety $\V$ of algebras which admits a $4$-ary term $m$ satisfying the equations
\begin{align*}
m(x,x,y,0) = x, \\
m(0,y,y,y) = y, \\
m(y,x,y,0) = y.
\end{align*}
For instance if $\V$ is a unital variety, i.e., $\V$ admits a binary operation $+$ satisfying $x + 0 = x = 0 + x$ then we may define $m(x,y,z,w) = x + w$. Or if $\V$ admits a majority term $p(x,y,z)$ (see the equations preceding Proposition~\ref{proposition: example majority category}), then we could define $p(x,y, z)$ then we could define $m(x,y,z,w) = p(x,y,z)$. We claim that $\V$ is strongly centralic. To see that it is centralic, suppose that $\theta$ is an equivalence relation on a product of algebras $X\times X$ in $\V$, and suppose that we are given $(x,0) \theta (y,0)$ then $\theta$ gives us
\begin{align*}
(x,0) \theta (y,0), \\
(x,z) \theta (x,z), \\
(y,z) \theta (y,z), \\
(0,z) \theta (0,z).
\end{align*}
Applying the operation $m$ component-wise above, gives us that $(x,z) \theta (y,z)$. To see that this variety satisfies $\T$: consider the diagram
\[
\xymatrix{
X \ar@/_/[dr]_f \ar[r]^-{(1_X,0)}& X \times X \ar@{->>}[d]_-{\phi} & X \ar[l]_-{(0,1_X)} \ar@/^/[ld]^f\\
& Y \\
}
\]
then we have that
\begin{align*}
\phi(x,y) &= \phi(m((0,y),(x,0),(x,y), (x,0)) \\
&= m(\phi(0,y),\phi(x,0),\phi(x,y),\phi(x,0)) \\
&= m(\phi(0,y),\phi(x,y),\phi(0,x),\phi(0,y)) = \phi(0,m(x,y,x,y)) = f(m(x,y,x,y)),
\end{align*}
so that $f$ is surjective since $\phi$ is. 
\begin{proposition}
Every pointed factor permutable category is strongly centralic.
\end{proposition}
\begin{proof}
By Remark~\ref{remark: factor permutable} if $\C$ is factor permutable then $\C$ is centralic. To see that it satisfies $\T$, consider the morphisms in the statement of $\T$. To see how $f$ is surjective let $(x,y) \in X \times X$ be any element. Then factor permutability gives a $z \in X$ such that the diagram of relations:
\[
\xymatrix{
(x,y) \ar@{-}[r]^{\pi_1} \ar@{..}[d]_{\phi} & (x,0) \ar@{-}[d]^{\phi}\\
(0,z)  \ar@{..}[r]_{\pi_1} & (0,x)
}
\]
holds. Then $f(z) = \phi(0,z) = \phi(x,y)$ --- so that $f$ is surjective.
\end{proof}

\section*{Concluding remarks}
We leave the investigation of what may be called \emph{locally (strongly) centralic} categories, i.e., finitely complete categories $\C$ where every category of points $\Pt_{\C}(X)$ is centralic (or strongly centralic), and what relation these categories have to the results of the papers \cite{Gra04, BournGran2004}, for a future work. We remark here that varieties of algebras which are locally centralic in this sense have been characterised in \cite{Hoefnagel2019b} by means of H.~P.~Gumm's shifting lemma: a variety $\V$ is locally centralic if and only if for any morphisms $p:A \rightarrow X$ and $q:B \rightarrow X$ in $\V$ and any congruence $\Theta$ on $A \times_X B$ we have $(x,u) \Theta (y, u) \Rightarrow (x,v) \Theta (y, v)$ for any elements $(x,u), (y, u), (x,v), (y, v)$ of $A \times_X B$. This can be seen as a special case of the shifting lemma:
\[
\xymatrix{
	(x,u) \ar@/^1.5pc/@{-}[r]^{\Theta} \ar@{-}[r]^{\Eq(p_2)} \ar@{-}[d]_{\Eq(p_1)} & (y,u) \ar@{-}[d]^{\Eq(p_1)} \\
	(x,v) \ar@/_1.5pc/@{..}[r]_{\Theta} \ar@{-}[r]_{\Eq(p_2)} & (y,v) 
}
\]
This observation was the main motivation behind this paper. We also remark here that some of the results in this paper took inspiration from the paper \cite{Bourn2021} --- see Proposition 2.2 (and its proof) therein. Note that categorical conditions similar to some of the equivalent statements of Proposition~\ref{proposition: reformulations of Ax1} have been investigated in  \cite{HoefnagelJanelidze-Nullital} as well as \cite{Hoefnagel2021}.

We also remark that the defining property of a finitely (bi)complete category to be centralic may be reformulated  as an exactness property as it is defined in \cite{Jacqmin2021} (see section 6.3 therein). Also, the general commutativity of a specified limit with a specified colimit is an exactness property amenable to the results of the paper \cite{JanelidzeJacqmin2021}.


\begin{thebibliography}{10}

\bibitem{AdamekRosicky2000}
J.~Ad\'amek and J.~Rosick\'y.
\newblock On sifted colimits and generalised varieties.
\newblock {\em Theory and Applications of Categories}, 8:33--53, 2001.

\bibitem{AdamekRosickyVitale2010}
J.~Ad\'amek, J.~Rosick\'y, and E.~M. Vitale.
\newblock What are sifted colimits?
\newblock {\em Theory and Applications of Categories}, 23:251--260, 2010.

\bibitem{BGO71}
M.~Barr, P.~A. Grillet, and {D. H. van} Osdol.
\newblock Exact categories and categories of sheaves.
\newblock {\em Springer, Lecture Notes in Mathematics}, 236, 1971.

\bibitem{BorceuxBourn}
F.~Borceux and D.~Bourn.
\newblock {\em Mal'cev, Protomodular, Homological and Semi-Abelian Categories}.
\newblock Kluwer, 2004.

\bibitem{Bou96}
D.~Bourn.
\newblock Mal'cev categories and fibration of pointed objects.
\newblock {\em Applied Categorical Structures}, 4:307--327, 1996.

\bibitem{Bou02}
D.~Bourn.
\newblock Intrinsic centrality and associated classifying properties.
\newblock {\em Journal of Algebra}, 256:126--145, 2002.

\bibitem{Bourn2005}
D.~Bourn.
\newblock Fibration of points and congruence modularity.
\newblock {\em Algebra Universalis}, 52:403--429, 2005.

\bibitem{Bourn2021}
D.~Bourn.
\newblock On congruence modular varieties and {G}umm categories.
\newblock {\em Communications in Algebra}, 50:2377--2407, 2021.

\bibitem{BournGran2004}
D.~Bourn and M.~Gran.
\newblock Categorical aspects of modularity.
\newblock {\em Fields Institute Communications}, 43:77--100, 2004.

\bibitem{EckmannHilton1962}
B.~Eckmann and P.~J. Hilton.
\newblock Group-like structures in general categories {I}.
\newblock {\em Mathematische Annalen}, 145:227--255, 1962.

\bibitem{FreydScendrov1990}
P.~Freyd and A.~Scendrov.
\newblock {\em Categories, Allegories}.
\newblock North-Holland Mathematical Library, 39, 1990.

\bibitem{Gra04}
M.~Gran.
\newblock Applications of categorical {G}alois theory in universal algebra.
\newblock {\em Fields Institute Communications}, 43:243 -- 280, 2004.

\bibitem{Gran2021}
M.~Gran.
\newblock An introduction to regular categories.
\newblock {\em In New perspectives in Algebra, Topology and Categories Coimbra
  Mathematical Texts, Vol. 1, Eds. M.M. Clementino, A. Facchini and M. Gran},
  Springer, 113 -- 145, 2021.

\bibitem{GrayPhD}
J.~R.~A. Gray.
\newblock Algebraic exponentiation and internal homology in general categories
  ({P}h{D} {T}hesis).
\newblock {\em University of Cape Town}, 2010.

\bibitem{Gray2012}
J.~R.~A. Gray.
\newblock Algebraic exponentiation in general categories.
\newblock {\em Applied Categorical Structures}, 20(6):543 -- 567, 2012.

\bibitem{Gum83}
H.~P. Gumm.
\newblock Geometrical methods in congruence modular algebras.
\newblock {\em Memoirs of the American Mathematical Society}, 45(286), 1983.

\bibitem{Hoe18a}
M.~Hoefnagel.
\newblock Majority categories.
\newblock {\em Theory and Applications of Categories}, 34:249--268, 2019.

\bibitem{Hoefnagel2019b}
M.~Hoefnagel.
\newblock Products and coequalisers in pointed categories.
\newblock {\em Theory and Applications of Categories}, 34:1386–1400, 2019.

\bibitem{Hoefnagel2020a}
M.~Hoefnagel.
\newblock Characterizations of majority categories.
\newblock {\em Applied Categorical Structures}, 28(1):113--134, 2020.

\bibitem{Hoefnagel2021}
M.~Hoefnagel.
\newblock Anticommutativity and the triangular lemma.
\newblock {\em Algebra Universalis}, 82:403--429, 2021.

\bibitem{Hoefnagel2023}
M.~Hoefnagel.
\newblock On the commutativity of products with coequalisers in general varieties of
  algebras ({\em preprint available}), 2023.

\bibitem{HoefnagelJanelidze-Nullital}
M.~Hoefnagel and Z.~Janelidze.
\newblock Isomorphism formulas for products and quotients in pointed categories
({\em in preparation}).

\bibitem{Huq1968}
S.~A. Huq.
\newblock Commutator, nilpotency, and solvability in categories.
\newblock {\em The Quarterly Journal of Mathematics}, 19(1):363--389, 1968.

\bibitem{Jacqmin2021}
P.~A. Jacqmin.
\newblock A class of exactness properties characterized via left {K}an
  extensions.
\newblock {\em Journal of Pure and Applied Algebra}, (106784), 2021.

\bibitem{JanelidzeJacqmin2021}
P.~A. Jacqmin and Z.~Janelidze.
\newblock On stability of exactness properties under the procompletion.
\newblock {\em Advances in Mathematics}, (107484), 2021.

\bibitem{Mac98}
S.~{Mac Lane}.
\newblock {\em Categories for the Working Mathematician}.
\newblock Graduate Texts in Mathematics. Springer New York, 1998.

\bibitem{Nelson2008}
N.~Martins-Ferreira.
\newblock Low-dimensional internal categorical structures in weakly {M}al’cev
  sesquicategories (PhD Thesis).
\newblock {\em University of Cape Town}, 2008.

\bibitem{Ferreira2008}
N.~Martins-Ferreira.
\newblock Weakly {M}al’cev categories.
\newblock {\em Theory and Applications of Categories}, 21:97--117, 2008.

\bibitem{Ferreira2012}
N.~Martins-Ferreira.
\newblock Weakly {M}al’tsev categories and distributive lattices.
\newblock {\em Journal of Pure and Applied Algebra}, 216:1961–1963, 2012.

\bibitem{Ferreira2015}
N.~Martins-Ferreira.
\newblock New wide classes of weakly {M}al’tsev categories.
\newblock {\em Applied Categorical Structures}, 23:741--751, 2015.

\end{thebibliography}
\end{document}